\documentclass[11pt]{article}
\usepackage{epsfig}
\usepackage{graphicx}
\usepackage{amssymb,amsmath}
\usepackage{caption}
\usepackage{array}
\newtheorem{ghazie}{{\bf Theorem}}[section]
\newtheorem{conj}[ghazie]{\bf Conjecture}
\newtheorem{lemma}[ghazie]{\bf Lemma}
\newtheorem{exm}[ghazie]{\bf Example}
\newtheorem{cor}[ghazie]{\bf Corollary}
\newtheorem{pro}[ghazie]{\bf Proposition}
\newcommand{\rem}[1]{\hspace*{-.0cm}\textbf{Remark.}\textit{#1}}
               
\newtheorem{preproof}{{\bf Proof.}}

\newenvironment{proof}[1]{\begin{preproof}{\rm
               #1}\hfill{$\rule{2mm}{2mm}$}}{\end{preproof}}

\begin{document}
\title{On the 1-2-3-conjecture}
\author{Akbar Davoodi\footnote{a.davoodi@math.iut.ac.ir} \ and Behnaz Omoomi\footnote{bomoomi@cc.iut.ac.ir}\\[.5cm]
Department of Mathematical Sciences\\
Isfahan University of Technology\\
84156-83111, Isfahan, Iran}
\date{}
\maketitle
\begin{abstract}
A $k$-edge-weighting of a graph $G$ is a function $w:E(G)\rightarrow\{1,2,\ldots,k\}$. An edge-weighting naturally induces a vertex coloring $c$, where for every $v\in V(G)$, $c(v)=\sum_{e\sim v}w(e)$. If the induced coloring $c$ is a proper vertex coloring, then $w$ is called a vertex-coloring $k$-edge weighting (VC$k$-EW). 
Karo\'{n}ski et al. (J. Combin. Theory Ser. B 91 (2004) 151--157) conjectured that every graph admits a VC$3$-EW. This conjecture is known as $1$-$2$-$3$-conjecture.
In this paper, first, we study the vertex-coloring edge-weighting of the cartesian product of graphs. Among some results, we prove that the $1$-$2$-$3$-conjecture holds for some infinite classes of graphs. Moreover, we explore some properties of a graph to admit a VC$2$-EW.
\end{abstract}

\section{\hspace*{-.6cm}. Introduction}
In this paper, we consider finite and simple graphs.
A \textit{$r$-vertex coloring} $c$ of $G$ is a function $c:V(G)\rightarrow\{1,2,\ldots ,r\}$. The coloring $c$ is called a \textit{proper vertex coloring} if for every two adjacent vertices $u$ and $v$, $c(u)\neq c(v)$. A graph $G$ is \textit{$r$-colorable} if $G$ has a proper $r$-vertex coloring.

A \textit{$k$-edge-weighting} of a graph $G$ is a function $w:E(G)\rightarrow\{1,2,\ldots,k\}$. An edge-weighting naturally induces a vertex coloring $c$, where for every $v\in~V(G)$, $c(v)=\sum_{e\sim v}w(e)$. The notion $e\sim v$, shows that $e$ is an edge incident to vertex $v$. If the induced coloring $c$ is a proper vertex coloring, then $w$ is called a \textit{vertex-coloring $k$-edge weighting} (VC$k$-EW). The minimum $k$ which $G$ has a VC$k$-EW is denoted by $\mu(G)$.

Obviously for a graph $G$ with components $G_{1}, G_{2}, \ldots, G_{t}$; $\mu(G)=\max\{\mu(G_{i})\,|\, 1\leq i\leq t\}$.
Also, note that the vertex-coloring $k$-edge-weighting is defined for graphs with $\Delta\geq 2$, where $\Delta$ is the maximum degree of vertices in $G$. Thus, we consider connected graphs with at least three vertices.

Karo\'{n}ski et al. in \cite{123} introduced the concept of vertex-coloring $k$-edge weighting and they proposed the following conjecture. 

\begin{conj}{\rm\cite{123}} {\rm($1$-$2$-$3$-conjecture)}
For every connected graph $G$ with at least three vertices, $\mu(G)\leq 3$.
\end{conj}

Addario-Berry et al. \cite{vc 30-ew} showed that for every connected graph with at least three vertices, $\mu(G)\leq 30$. Then Addario-Berry et al. \cite{vc 16-ew} improved the bound to 16. Later, Wang and Yu \cite{vc 13-ew} improved this bound to 13. Recently, Kalkowski et al. \cite{vc 5-ew} showed that for every connected graph $G$ with at least three vertices, $\mu(G)\leq 5$.

It is proved that for every $3$-colorable graph $G$, $\mu(G)\leq3$ \cite{123}. In general the similar fact for $2$-colorable graphs is not true. The only known families of bipartite graphs with $\mu(G)=3$ are cycles $C_{4k+2}$ and theta graph $\theta(1,4k_2+1,4k_3+1,\ldots,4k_r+1)$, where $\theta(1,4k_2+1,4k_3+1,\ldots,4k_r+1)$ is a graph obtained by $r$ simple paths of length $4k_i+1$ with common end vertices \cite{chang}. In the following two theorems some sufficient conditions for bipartite graphs to admit a VC$2$-EW are presented.

\begin{ghazie}{\rm\cite{chang}}\label{delta=1}
If $G\ncong K_{2}$ is a connected bipartite graph with one of the following conditions, then $\mu(G)\leq 2$.
\begin{itemize}
\item
$\delta(G)=1$, where $\delta(G)$ is the minimum degree of $G$.
\item
$G$ has a part with even number of vertices.
\end{itemize}
\end{ghazie}
The notion $N[v]$ denotes the $N(v)\cup \{v\}=\{u\,|\,u\in V(G), uv\in E(G)\}\cup\{v\}$.

\begin{ghazie}{\rm\cite{europ}}\label{G-N[v]}
Let $G\ncong K_{2}$ be a connected bipartite graph. If one of the following conditions holds, then $\mu(G)\leq2$.
\begin{itemize}
\item There exists a vertex $v$ such that $\deg(v)\notin\{\deg(u) | u\in N(v)\}$ and $G-N[v]$ is connected.
\item There exists a vertex $v$ of degree $\delta(G)$ such that
$\deg(v)\notin\{\deg(u) | u\in N(v)\}$ and $G-v$ is connected.
\item $G$ is $3$-connected.
\end{itemize}
\end{ghazie}
\begin{ghazie}{\rm\cite{chang}}\label{P_n & C_n}
Let $P_n$, $C_n$ and $K_n$, $n\geq3$, denote the path, cycle and complete graph with $n$ vertices, respectively. Then
\begin{equation*}
\mu(P_n)=
 \left\{
    \begin{array}{ll}
	1 & n=3 \\ 
	2 & n\geq 4
    \end{array}
 \right.
,\quad
\mu(C_n)=
 \left\{
    \begin{array}{ll}
	2 & n\equiv 0\hspace*{-.35cm}\pmod4 \\ 
	3 & \text{otherwise}
    \end{array}
 \right.
 \text{,and }\ 
 \mu(K_n)=3.
\end{equation*}
\end{ghazie}
The structure of this paper is as follows. In Section \ref{sec: cartesian}, we consider the $1$-$2$-$3$-conjecture for the cartesian product of graphs. First, we prove that the conjecture holds for the cartesian product of some well known families of graphs. Moreover, we prove that the cartesian product of connected bipartite graphs admits a VC$2$-EW. Then, we determine $\mu(G)$ for some well known families of graphs. In \cite{vc 16-ew} Addario-Berry et al. proved that almost all graphs admit a VC$2$-EW. The problem of classifying graphs which admit a VC$2$-EW is an open problem.
In Section \ref{sec: VC2-EW}, first, we present some sufficient conditions for a graph to admit a VC$2$-EW. Then, in Section \ref{sec: VCEW of bipartite} following to exploring more bipartite graphs with $\mu(G)=3$, we consider a generalization of theta graphs which are called generalized polygon trees and determine $\mu(G)$ for some bipartite generalized polygon trees.

\section{\hspace*{-.6cm}. 1-2-3-conjecture for the cartesian product of graphs}\label{sec: cartesian}
The \textit{cartesian product} of two graphs $G$ and $H$, written $G\Box H$, is the graph with vertex set $V(G)\times V(H)$ specified by putting $(u,v)$ adjacent to $(u', v')$ if and only if $u=u'$ and $vv'\in E(H)$, or  $v=v'$ and $uu'\in E(G)$. 

In this section, we prove that if the $1$-$2$-$3$-conjecture holds for two graphs $G$ and $H$, then it also holds for $G\Box H$. Moreover, we prove that if $G$ and $H$ are bipartite graphs, then $\mu(G\Box H)\leq 2$.
\begin{ghazie}\label{max}
For every two graphs $G$ and $H$, $\mu(G\Box H)\leq \max\{\mu(G),\mu(H)\}$.
\end{ghazie}
\begin{proof}{
Let $k:=\max\{\mu(G),\mu(H)\}$ and $w_{1}:E(G)\rightarrow \{1,2,\ldots,k\}$, $w_{2}:E(H)\rightarrow \{1,2,\ldots,k\}$ be VC$k$-EW for graphs $G$ and $H$, respectively. We define $w:E(G\Box H)\rightarrow \{1,2,\ldots,k\}$, $w((u,v)(u',v))=w_{1}(uu')$, and $w((u,v)(u,v'))=w_{2}(vv')$, where
$u,u'\in V(G)$, $v,v'\in V(H)$. It is easy to see that $w$ is a VC$k$-EW for $G\Box H$
}\end{proof}
%

By Theorem \ref{P_n & C_n} and Theorem \ref{max}, we conclude that the $1$-$2$-$3$-conjecture is true for some well known families of graphs.  
\begin{cor}
The $1$-$2$-$3$-conjecture is true for the cartesian product of complete graphs, cycles and paths. For instance the following graphs satisfy the $1$-$2$-$3$-conjecture.\\
$\bullet\; C_{n}\Box P_{m}$\hspace{2.75cm}
$\bullet\; K_{n}\Box P_{m}$\hspace{2.6cm}
$\bullet\; K_{n}\Box C_{m}$\hspace{2cm}
\\
$\bullet\; C_{n_1}\Box C_{n_2}\Box\cdots\Box C_{n_k}$\hspace{.7cm}
$\bullet\; P_{n_1}\Box P_{n_2}\Box\cdots\Box P_{n_k}$\hspace{.7cm}
$\bullet\; K_{n_1}\Box K_{n_2}\Box\cdots \Box K_{n_k}$
\end{cor}
We need the following lemma to prove our main theorem in this section. In what follows $n(G)$ and $\kappa(G)$ denote the order and vertex connectivity of $G$, respectively.
\begin{lemma}\label{kapa} {\rm\cite{simon}}
For every two graphs $G$ and $H$, $$\kappa(G\Box H)=\min\{\delta(G\Box H),\kappa(H).n(G),\kappa(G).n(H)\}.$$
\end{lemma}
In the following theorem we prove that the cartesian product of every two bipartite graphs admits a VC$2$-EW. 
\begin{ghazie}\label{main result}
If $G$ and $H$ are two connected bipartite graphs and $G\Box H\neq K_2$, then $\mu(G\Box H)\leq 2$.
\end{ghazie}
\begin{proof}{
It can be seen that $G\Box H$ is a connected bipartite graph if and only if $G$ and $H$ are connected bipartite graphs.
If  $n(G)=n(H)=2$, then $G\cong C_4$ and by Theorem~\ref{P_n & C_n}, $\mu(G)=2$.
If $n(G), n(H)>2$ and both have a leaf vertex, then by Theorem \ref{delta=1} 
and Theorem~\ref{max}, $\mu(G\Box H)\leq 2$, otherwise $\delta(G)+\delta(H)\geq 3$. Hence, by Lemma \ref{kapa}, $\kappa(G\Box H)\geq~3$ and hence by Theorem \ref{G-N[v]}, $\mu(G\Box H)\leq 2$.

Now, let $n=n(G)>2$ and $n(H)=2$. Thus, $G\Box H\cong G\Box K_2$. Let $G_1$ and $G_2$ be the induced subgraphs representing the first and the second copy of $G$, respectively. To give a VC$2$-EW for $G\Box K_2$, first we assign weight $1$ to all the edges in $G_1$ and weight $2$ to all the edges in $G_2$. We denote the unweighted edge $e$ incident to vertex $u\in V(G_1)$ by $e_u$. Thus, for every two adjacent vertices $u$ and $v$, where $u\in G_1$, and $v\in G_2$, independent to the weight of $e_u$, we have $c(u)\neq c(v)$. Now we assign a proper weight to the unweighted edges such that for every $uv\in E(G_1)$, $c(u)\neq c(v)$, that also implies $c(u)\neq c(v)$ for every edge $uv\in E(G_2)$. We do this among the following process. 

Let $U_{i}=\{u\in V(G_{1})\: |\: d_{G_{1}}(u)=i\}$, $1\leq i\leq n-1$. Now consider the partition $C=\{A\: |\: A\ \text{is a component of}\; \langle U_{i}\rangle,\; 1\leq i\leq n-1\}$ of the vertices of $G_1$. Note that to get a proper weighting, for each $u\in A$ the weight of $e_u$ forces the weight of the other edges incident to the vertices in $A$ as follows. For each unweighted  edge $e_v,\; v\in A$, let
\begin{equation}
w(e_{v})=
 \left\{
    \begin{array}{ll}
	w(e_u)+1\hspace*{-.3cm}\pmod 2 & d(u,v) \ is \ odd \\[.2cm] 
	w(e_u) & \text{otherwise,}
    \end{array}
 \right.\tag{$*$}
\end{equation}
where $d(u,v)$ is the length of a shortest path between $u$ and $v$ in $G$. Thus, for every component $A$, it is enough to assign the weight of $e_u$ only for one vertex $u\in A$.

Note that, by $(*)$,  for every edge $uv\in E(G_1)$, where $d_{G_1}(u)=d_{G_1}(v)$, $c(u)-c(v)=w(e_u)-w(e_v)\neq0$. Also if $|d_{G_1}(u)-d_{G_1}(v)|\geq2$, then independent to the weights of $e_u$ and $e_v$, we have $c(u)\neq c(v)$.

The following algorithm provide a desired edge weighting for $e_u$, $u\in A$.\\[.1cm]
\hspace*{.2cm}\begin{tabular}{|p{14cm}|}
\hline 
\tt \hspace*{.2cm}10 \hspace*{.3cm}	${\tt  h=0}$;\\[.3cm]
\hspace*{.2cm}20	\hspace*{.45cm}	{\tt while}  ($\mathtt{C\neq \emptyset}$)\\[.3cm]
\hspace*{.2cm}30	\hspace*{.9cm}	h=h+1, {\tt choose an arbitrary component} $\mathtt{A\in C}$ {\tt and define} $\mathtt{{S^{(h)}=\{A\}}}$,\\ \hspace*{1.7cm}$\mathtt{C=C\backslash A,\; p(A)=\emptyset}$ {\tt and let} $\mathtt{w(e_{u})=1}$ {\tt for an arbitrary vertex} $\mathtt{u\in A}$;\\[.2cm]
\hspace*{.2cm}40	\hspace*{26pt}	{\tt while} $\left(\begin{array}{l}
{\tt \text{\tt there exist}\ A'\in C,\ A\in S^{(h)}\ \text{\tt and egde}\ e=xy,\ \text{\tt where}\ x\in A}\\
{\tt \text{\tt and}\ y\in A'\ and\ |d_{G_{1}}(x)-d_{G_{1}}(y)|=1}
\end{array} \newline \right)$\\[.4cm]
\hspace*{.2cm}50	\hspace*{1.8cm}	${\tt S^{(h)}=S^{(h)}\mathtt{\cup} {A'}}$, ${\tt C=C\backslash A'}$, ${\tt w(e_{y})=w(e_{x})}$, ${\tt p(A')=A}$;\\ 
\hline 
\end{tabular}

To complete the proof it is enough to show that for every two adjacent vertices $u, v$, where $|d_{G}(u)-d_{G}(v)|=1$, $c(u)\neq c(v)$. Without loss of generality assume that $u\in A$, $v\in A'$ and $d_{G}(u)=d_{G}(v)-1$.

If $p(A')=A$, then there exist $x\in A$ and $y\in A'$ such that $xy\in E(G)$ and $w(e_x)=w(e_y)$. Since $G$ is bipartite $d_{A}(x,u)+d_{A'}(y,v)$ is even, thus $d_A(x,u)$ and $d_A'(y,v)$ both are even or both are odd. Therefore by $(*)$, $w(e_{u})=w(e_{v})$ and since $\deg(u)\neq \deg(v)$ we have $c(u)\neq c(v)$.

If $p(A')\not =p(A)$, then by the given algorithm, there exist $A=A_1, A_2, \ldots, A_r=~A'$ in $S^{(h)}$, for some $h$, where $p(A_{i+1})=A_{i},\; 1\leq i\leq r-1$. Moreover, for $i$, $1\leq i\leq r-1$, there exist the edges $x_iy_{i+1}$ such that $x_i, y_{i}\in~A_i$, $y_r\in~A_r$ and $w(e_{x_i})=w(e_{y_{i+1}})$.
Note that for $i$, $1\leq i\leq r-1$, $|\deg(x_i)-\deg(y_{i+1})|=1$ and $\deg(x_1)=\deg(y_r)-1$. Therefore relation $(*)$ implies that $r$ is even. On the other hand, since $G$ is bipartite, $d(u,x_1)+\sum_{i=2}^{r-1} d(y_{i},x_{i})+d(y_{r},v)+r$ is even; otherwise, we get an odd cycle in $G$. Therefore, $d(u,x_1)+\sum_{i=2}^{r-1} d(y_{i},x_{i})+d(y_{r},v)$ is even. By the equality $w(e_{x_i})=w(e_{y_{i+1}})$ and the relation $(*)$, we get $w(e_u)=w(e_v)$ which completes the proof.
}\end{proof}
Obviously, for every nontrivial graph $G$, $\mu(G)=1$ if and only if $G$ has no adjacent vertices with a same degree.
\begin{pro}\label{mu=1 for cartesian}
For every two graphs $G$ and $H$, $\mu(G\Box H)=1$ if and only if $\mu(G)=1$ and $\mu(H)=1$.
\end{pro}
\begin{proof}{
By Theorem \ref{max}, the condition is sufficient. Conversely, if $\mu(G\Box H)=1$, then $d_{G\Box H}(u,v)\neq d_{G\Box H}(u',v)$, where $uu'\in E(G)$. 
 On the other hand $d_{G\Box H}(u,v)=d_{G}(u)+d_{H}(v)$. Therefore, for every edge $uu'\in E(G)$, $d_{G}(u)\neq d_{G}(u')$. This implies $\mu(G)=1$. Similarly, $\mu(H)=1$.
}\end{proof}
By Theorem \ref{main result} and Proposition \ref{mu=1 for cartesian}, we have the following corollary.
\begin{cor}
For every two connected bipartite graph $G$ and $H$, which at least one of them doesn't admit a VC$1$-EW, we have $\mu(G\Box H)=2$. Particularly, for even cycles $C_n$ and $C_m$, $\mu(C_{n}\Box C_{m})=2$, $\mu(C_n\Box P_m)=2$. Also $\mu(P_{n}\Box P_{m})=2$ and for the cube graph $Q_n=K_2\Box K_2\Box\cdots\Box K_2$, we have  $\mu(Q_{n})=2$.
\end{cor}
\section{\hspace*{-.6cm}. Vertex-coloring $2$-edge-weighting}\label{sec: VC2-EW}
In this section, following to investigating the properties of graphs which admit a VC$2$-EW, first we present some sufficient conditions for a graph to admit a VC$2$-EW.
In the following, we consider separable graphs.
\begin{ghazie}\label{2 Blocks}
Let $G_1$ and $G_2$ be  simple connected graphs which $V(G_1)\cap V(G_2)=\{v\}$. If one of the following conditions holds, then $\mu(G)\leq 2$.
\begin{itemize}
\item
$d_G(v)\geq4$ and $\forall i,\; 1\leq i\leq 2$, $\mu(G_i)\leq 2$ and $\forall u\in N(v)$, $\deg(u)\leq 2$.
\item
$d_G(v)\geq4$ and $\mu(G_1)\leq 2$ and $\forall u\in N(v)$, $\deg(u)\leq 2$ and $G_2$ is a cycle.
\item
$G_1$ and $G_2$ are cycles.
\end{itemize}
\end{ghazie}
\begin{proof}
{
Assume the first condition holds. Since $\mu(G_i)\leq 2$, there exist vertex-coloring $2$-edge-weightings $w_1:E(G_1)\rightarrow\{1,2\}$ and $w_2:E(G_2)\rightarrow\{1,2\}$. Let $w:E(G)\rightarrow\{1,2\}$ be defined as follows
$$w(e)=
\left\lbrace 
\begin{array}{ll}
w_1(e) & e\in E(G_1), \\[.2cm]
w_2(e) & e\in E(G_2).
\end{array}
\right.$$
Then $c(v)\geq4$ and for every vertex $u\in N(v)$, $c(u)\leq 4$.
Now if there exists a vertex $u\in N(v)$, such that $c(u)=c(v)$, then $c(u)=4$ implies $w(uv)=2$ while $c(v)=4$ implies $w(uv)=1$, a contradiction.
Hence, $w$ is a desired edge-weighting.

Now, suppose the second condition holds, and $w_1$ be a vertex-coloring $2$-edge weighting of $G_1$. For each edge $e\in E(G_1)$ assign the weight $w_1(e)$. For other edges start from an incident edge on $v$, assign the weights $2,2,1,1$ periodically, such that the last edge be incident on $v$. It is easy to check that this is a vertex-coloring $2$-edge-weighting.

If the last condition holds. For each $G_i$, start from an incident edge on $v$, assign the weights $2,2,1,1$ periodically such that the last edge be incident on $v$. Clearly for every two adjacent vertices $x$ and $y$ which $x,y\notin \{v\}$, $c(x)\neq c(y)$.
For every $u\in N(v)$, $c(u)\in\{2,3,4\}$ while $c(v)\geq6$. Therefore, $\mu(G)=2$.
}
\end{proof}
\begin{cor}\label{k Blocks}
Let $G$ be a simple connected graph with the blocks $B_1,B_2,\ldots ,B_r$,\ $r\geq 2$, and cut vertices $v_1,v_2,\ldots ,v_{t}$. If for every i, $1\leq i\leq r$, $B_i$ satisfies one of the following conditions, then $\mu(G)\leq2$.
\begin{itemize}
\item $B_i$ is a cycle.
\item  $\mu(B_i)\leq2$ and $\forall v_j\in V(B_i),\; \forall u\in N(v_j),\; \deg(u)\leq 2$.
\end{itemize}
\end{cor}
\begin{ghazie}\label{thm: MSP}
Let $G\ncong C_{4k+r}$ for $r=1,2,3$ be a connected graph. If one of the following conditions holds, then $\mu(G)\leq2$.
\begin{itemize}
\item $G$ contains no edge as a maximal simple path and for every maximal simple path of $G$ which is of length $1\pmod4$ connecting two vertices $x$ and $y$, $\{\deg(x),\deg(y)\}\neq\{3\}$.
\item $G$ contains no maximal simple path of length $3\pmod4$ and for every edge $e=xy$ which is a maximal simple path of $G$, $\deg(x)\neq\deg(y)$.
\end{itemize}
\end{ghazie}
\begin{proof}
{
The edge set of $G$ is disjoint union of maximal simple paths. We assign the weights $1$ and $2$ to the edges of every maximal simple path, starting from one of the end-edges, according to the pattern $1,1,2,2$ and the pattern $2,1,1,2$ periodically, for maximal simple paths of length $2$ and length $0,1$ or $3$, respectively. Now, if there exists an edge $uv\in E(G)$ such that $c(u)=c(v)$, then it is clear that  this edge is an end-edge for a maximal simple path. Therefore one of the vertices $u$ or $v$ is an end-vertex of a maximal simple path. Without loss of generality assume that $v$ is such a vertex.

If $\deg(v)=1$, then $c(v)<c(u)$. If $\deg(v)=2$, then $G\cong C_{4k}$ and it is easy to see that the given pattern is a VC$2$-EW for $G$. Thus, we assume that $\deg(v)\geq3$.

Assume the first condition holds. Since for every vertex $u\in N(v)$, $\deg(u)=2$ and by the given patterns, we have $c(u)\in\{2,3,4\}$. Clearly if $\deg(v)>4$, then $c(v)>4$ and $c(u)\neq c(v)$. On the other hand, if $\deg(v)=c(v)=4$, then by the given patterns, $c(x)=2$ for every $x\in N(v)$. Thus we assume that $\deg(v)=3$ and there is two possibilities $c(v)=3$ and $c(v)=4$.
If $c(v)=3$, then by the given patterns, $c(x)=2$ for every $x\in N(v)$. Let $c(v)=4$. Thus one of the edges on $v$ has weight $2$ and by the given patterns this edges belongs to a maximal simple path of length $1\pmod4$. Replace the pattern of this path from $2,1,1,2$ to $1,1,2,2$, Since the vertex which is on other end of this path is not of degree $3$, this process reduce the number of edges with two end of same color. Hence, repeated applications give the desired weighting.

If the last condition holds, we change the pattern of maximal simple paths of length $2\pmod4$ to $2,2,1,1$. Since every edge incident to $v$ has weight $2$, we have $c(v)=2\deg(v)\geq6$ and if $e=uv$ is a maximal simple path of $G$, by the assumption, $c(u)\neq c(v)$. we assume that $u$ is an internal vertex for a maximal simple path. Hence, $\deg(u)=2$ and by the given patterns, we have $c(u)\in\{3,4\}$. Therefore $c(u)\neq c(v)$.
}
\end{proof}
\begin{cor}\label{cor: MSP}
Let $G\ncong C_{4k+r}$ for $r=1,2,3$ be a connected graph. If one of the following conditions holds, then $\mu(G)\leq2$.
\begin{itemize}
\item $G$ contains no maximal simple path of length $1\pmod 4$.
\item $G$ contains no maximal simple path of length $3\pmod4$ and no edge as a maximal simple path.
\item Every maximal simple path of $G$ is of length $4k+1$, for $k\geq1$.
\end{itemize}
\end{cor}
\begin{cor}\label{cor: p and p'}
Let $G\ncong C_{4k+r}$ for $r=1,2,3$ be a connected graph contains no edge as a maximal simple path. If $\mu(G)>2$, then $G$ has $P$ and $P'$ as a maximal simple path of length $1\pmod4$ and $3\pmod4$, respectively.
\end{cor}
\rem{
In every VC$2$-EW of path $P_n$, if $n\equiv 1\pmod 4$, two end-edges of $P_n$ get a same weight and if $n\equiv 3\pmod 4$, they get different weights. Furthermore, if $n\not\equiv1,3\pmod4$, then for every arbitrary assignment to end-edges of $P_n$ with $\{1,2\}$, there is a VC$2$-EW for $P_n$.
}

The same fact for the remaining cases, that every maximal simple path of $G$ is of length $1$ or for example $G$ has maximal simple paths of length $1$ and $3$, is not true.
The following proposition gives an example of graphs in which every maximal simple path is of length $1\pmod4$ and $\mu(G)=2$, while $\mu(K_n)=3$.
Also, Figure \ref{exm1} is an example of graphs which contain maximal simple path of length $1\pmod4$ and $3\pmod4$ with $\mu(G)=3$.

\begin{pro}\label{equitable}
For every complete $r$-partite graph $K_{n,n,\ldots ,n}$ with $r,n\geq 2$, $\mu(K_{n,n,\ldots ,n})=~2$.
\end{pro}
\begin{proof}
{Let
\begin{equation*}
A=\left(%
\begin{array}{cccc}
0 & B & \cdots & B \\
B^{T} & \ddots & \ddots & \vdots \\
\vdots & \ddots & \ddots & B \\
B^{T} & \cdots & B^{T} & 0 \\
\end{array}%
\right)_{rn\times rn}
\end{equation*}
be the weighted adjacency matrix of $G=K_{n,n,\ldots ,n}$, where
$$B=\left(%
	\begin{array}{cccc}
		1 & 1 & \ldots & 1 \\
		\vdots & \vdots & & \vdots \\
		1 & 1 & \ldots & 1 \\
		2 & 2 & \ldots & 2 \\
	\end{array}
\right)_{n\times n}$$
Note that the induced color on every vertex is the sum of entries in it's corresponding row in $A$. Thus, for every two arbitrary vertices $u$ and $v$ in $i$-th part and $j$-th part of $G$, $c(u)\in\{(r-i)n+(i-1)(n+1),2(r-i)n+(i-1)(n+1)\}$ and $c(v)\in\{(r-j)n+(j-1)(n+1),2(r-j)n+(j-1)(n+1)\}$. Since $i\neq j$ we have
\begin{equation*}
 (r-i)n+(i-1)(n+1) \neq (r-j)n+(j-1)(n+1)
\end{equation*}
and
\begin{equation*}
2(r-i)n+(i-1)(n+1)\neq 2(r-j)n+(j-1)(n+1)
\end{equation*}
On the other hand, $j\leq r$ hence, $j-i\leq n(r-i)$ with equality if and only if $n=1$ and $j=r$. Thus $(r-j)n+(j-1)(n+1)\neq 2(r-i)n+(i-1)(n+1)$. Similarly $(r-i)n+(i-1)(n+1)\neq 2(r-j)n+(j-1)(n+1)$. Therefore, $A$ is a desired edge weighting. 
}
\end{proof}
\begin{exm}\label{exm:mu=3}
By the Remark, it's easy to check that the graph is shown in Figure \ref{exm1} is a bipartite graph with $\mu(G)=3$, in which $p_{i}$, $q_{i}$ and $r_{i}$ are paths of length $1$, $1$ and $3\pmod 4$, respectively.
\begin{figure}[ht]
\centering
{\includegraphics[scale=.5]{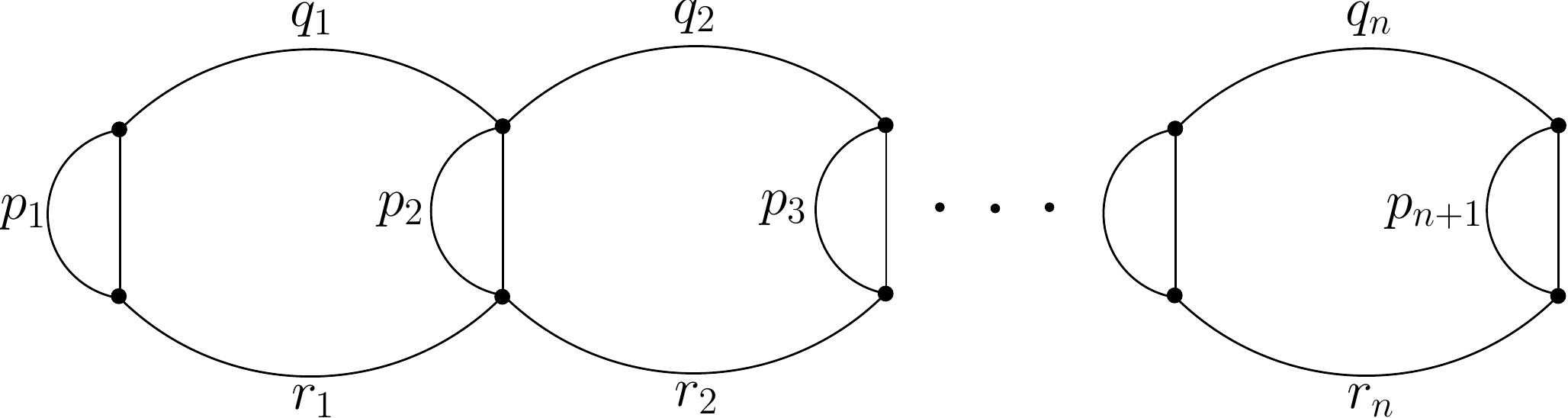} }
\caption{A family of $2$-connected bipartite graphs with $\mu(G)=3$\label{exm1}}
\end{figure}
\end{exm}
We end this section with the following theorem which is an improvement of Theorem \ref{thm: MSP}.
\begin{ghazie}\label{thm: our conj}
If $G\ncong C_{4k+r}$ for $r=1,2,3$ is a connected graph contains no edge as a maximal simple path, then $\mu(G)\leq2$.
\end{ghazie}
\section{\hspace*{-.6cm}. Vertex-coloring edge-weighting of bipartite graphs}\label{sec: VCEW of bipartite}
It is known that every $3$-colorable graph admits a VC$3$-EW \cite{123}, and there are bipartite graph $G$ with $\mu(G)=2$ and also bipartite graph $G$ with $\mu(G)=3$. 
The only known bipartite graphs with $\mu(G)=3$ are $C_{4k+2}$ and the theta graphs $\theta(1,4k_2+1,\ldots,4k_r+1)$ given in \cite{chang}. 
The \textit{theta graph} $\theta(l_{1}, l_{2},\ldots , l_{r})$ is the graph obtains from $r$ disjoint paths, of lengths $l_{1}, l_{2},\ldots , l_{r}$, respectively, by identifying their end-vertices called \textit{the roots} of the graph. Notice that $\theta(l_{1})=P_{1+l_{1}}$ and $\theta(l_{1},l_{2})=C_{l_{1}+l_{2}}$.
In Example \ref{exm:mu=3}, we constructed more bipartite graphs with $\mu(G)=3$. The natural question is whether there exists another well known family of bipartite graphs with $\mu(G)=3$? Regarding to answer to this question, in this section, we consider a generalization of theta graphs. The theta graphs are special cases of well known family of graphs called generalized polygon trees.
First, in the following theorem we give a sufficient condition to guarantee that a bipartite graph admit a VC$2$-EW. 
\begin{ghazie}\label{d(v)>d(u)}
If $G$ is a connected bipartite graph contains a vertex $v$ such that $\deg(v)>\deg(u)$ for every $u\in N(v)$ and $G-v$ is connected, then $\mu(G)\leq 2$.
\end{ghazie}
\begin{proof}{
Let $U$ and $W$ be parts of the graph $G$. If $|U|.|W|$ is even, by Theorem \ref{delta=1}, the result follows. Thus we assume that both $|U|$ and $|W|$ are odd. Let $v\in U$ satisfy $\deg(v)>\deg(u)$ for every $u\in N(v)$ and $G-v$ is connected. Since $|U-v|$ is even, by Theorem \ref{delta=1}, $G-v$ admits a VC$2$-EW such that $c(x)$ is odd for every $x\in U-v$ and $c(y)$ is even for every $y\in W$. Now we assign weight $2$ to all the edges that are adjacent to vertex $v$. 
Clearly $c(x)$ is odd for every $x\in U-v$ and $c(y)$ is even for every $y\in W$. Also $c(v)=2d_G(v)>2d_G(u)\geq c(u)$ for every $u\in N(v)$.
}
\end{proof}
\begin{ghazie} {\rm\cite{chang}}\label{theta}
Let $G=\theta(l_{1}, l_{2},\ldots , l_{r})$, where $r\geq 3$ and $1\leq l_{1}\leq l_{2}\leq\cdots\leq l_{r}$. Then,
\begin{center}
$\mu(G)=
 \left\{
    \begin{array}{ll}
	1 & l_i=2, \; \forall i \\ 
	3 & l_1=1 \; \text{and} \;\; l_i\equiv 1\pmod4 \; \forall i\not =1\\
    2 & {\text{otherwise.}}
    \end{array}
 \right.$
\end{center}
\end{ghazie}
A \textit{generalized polygon tree} is the graph defined recursively as follows. A cycle $C_{p}$ $(p\geq 3)$ is a generalized polygon tree. Next, suppose $H$ is a generalized polygon tree containing a simple path $P_{k}$, where $k\geq 1$. If $G$ is a graph obtained from the union of $H$ and a cycle $C_{r}$, where $r>k$, by identifying $P_{k}$ in $H$ with a path of length $k$ in $C_{r}$, then $G$ is also a generalized polygon tree.\cite{koh}

In the following theorem we consider bipartite generalized polygon trees with at most three interior regions, (see Figure \ref{gpt}) and by determining the $\mu(G)$, we classify that in which conditions they admit a VC$2$-EW.
\begin{ghazie}\label{thm:GPTs}
Let G be a bipartite generalized polygon tree with at most three interior regions shown in Figure {\rm\ref{gpt}}, which $p_{1},p_{2},\ldots ,p_{6}$ are paths of length $a$, $b$, $c$, $d$, $e$ and $f$, respectively. Then $\mu(G)=1$ if and only if one of the following conditions holds
\begin{itemize}
\item $a=b=c=d=e=f=2$.
\item $a=b=e+f=2$ and $c=d=0$.
\item $a=d=f=2$, $b=c=1$ and $e=0$.
\item $a=d=f=2$, $b=c=2$ and $e=0$.
\item $a=b=c=d=2$ and $e=f=0$.
\end{itemize}
And $\mu(G)=3$ if and only if one of the following conditions holds
\begin{itemize}
\item $a=b=c=d=0$ and $e+f\equiv 2\pmod 4$.
\item $b=1$, $c=d=0$ and $a,e+f\equiv1\pmod4$.
\item $b=1$, $e=f=0$ and $a,c,d\equiv 1\pmod 4$.
\item $b=c=1$, $a,d,e\equiv 1\pmod 4$ and $f\equiv 3 \pmod 4$.
\end{itemize}
Otherwise $\mu(G)=2$.
\begin{figure}[ht]
\centering
{\includegraphics[scale=.4]{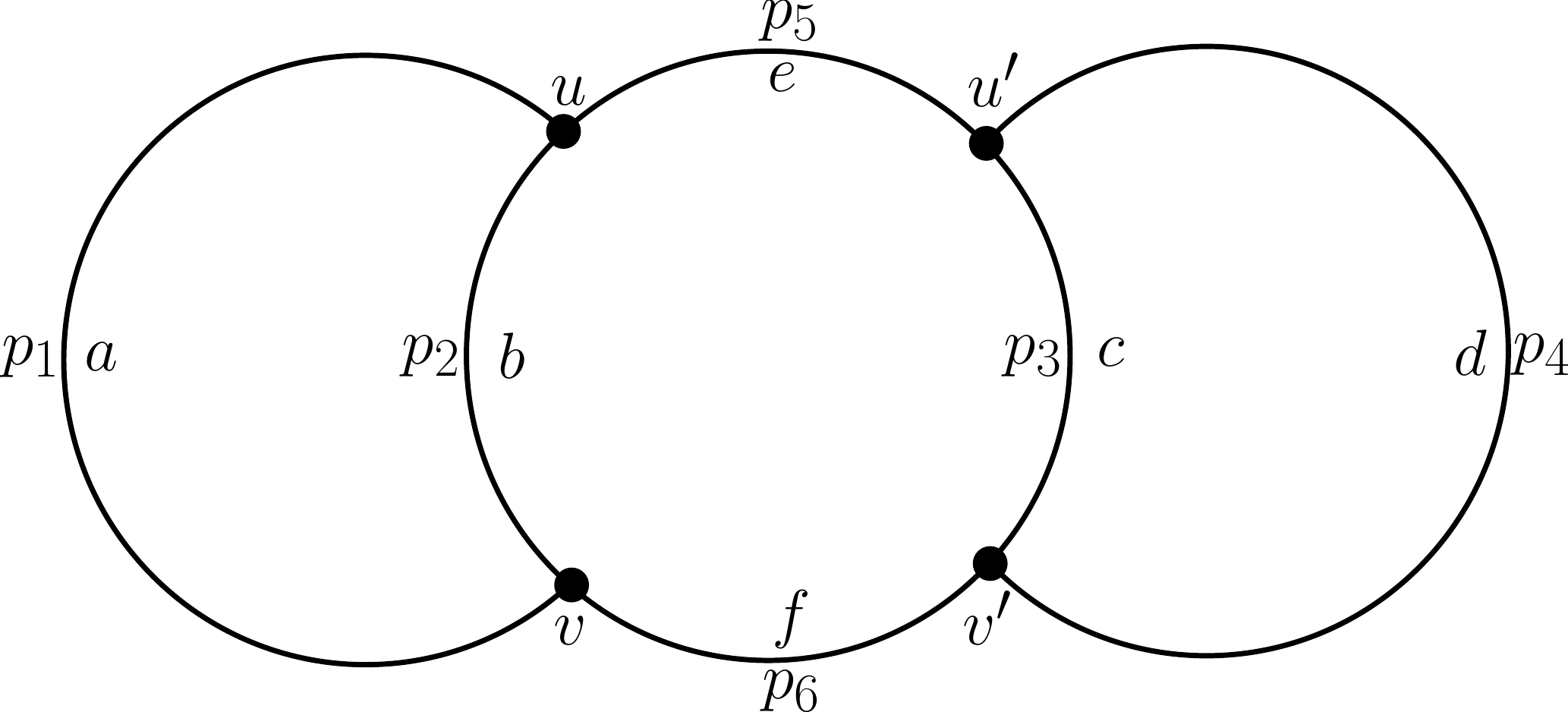}}
\caption{A generalized polygon tree with three interior regions\label{gpt}}
\end{figure}
\end{ghazie}
\begin{proof}
{
We know that a graph admits a vertex-coloring 1-edge-weighting if and only if every two adjacent vertices have different degrees. It can be verify that this condition holds if and only if one of the given possibilities in the statement occurs.

For $b$ and $c$ there is three possibilities $b=c=0$ or $b\neq 0$, $c=0$ or $b,c\neq 0$.

Let $b=c=0$. If $a=d=0$, then $G$ is an even cycle $C_{e+f}$ and by Theorem \ref{P_n & C_n}, if $e+f\equiv 0\pmod 4$, then $\mu(G)=2$, otherwise, $\mu(G)=3$. If $a\neq 0$ or $d\neq 0$, by Corollary~\ref{k Blocks}, $\mu(G)=2$.

Now let $b\neq 0$ and $c=0$. If $d=0$ then $G\cong \theta(a,b,e+f)$, and we have done. For $d\neq 0$, $G$ is a graph with two blocks $\theta(a,b,e+f)$ and cycle $C_d$.
Thus by Theorem \ref{2 Blocks}, $\mu(G)=2$ unless  $\mu(\theta(a,b,e+f))=3$. In this case we give a VC$2$-EW for $G$  in Figure $3$(\ref{last}). In the figures the edges and paths are denoted by straight lines and curves, respectively, and every path is weighted periodically by the given pattern through the denoted direction. 

Otherwise $b,c\neq 0$. In this case, we consider three possibilities $b,c>1$ or $b=1$, $c>1$ or $b=c=1$.

Now let $b,c>1$. In this case if $e\neq 1$ or $f\neq 1$, then there exists a vertex $x$ with $\deg(x)\geq 3$ such that $\deg(x)>\deg(y)$ for all $y\in N(x)$ and $G-x$ is connected. Thus by Theorem \ref{d(v)>d(u)},
 $\mu(G)\leq 2$. Therefore we assume that $e=f=1$. Notice that if $a=0$ or $d=0$, then $G$ is a theta graph and we have done. Now, since $G$ is bipartite $b+c$ is even. Hence, $b+c\equiv 0\pmod4$ or $b+c\equiv 2\pmod4$. For the first possibility if there is a part of even order, by Theorem \ref{delta=1},
we have done. Otherwise we have two odd parts and $a+d\equiv2\pmod4$. In this case we have four possibilities $a\equiv2\pmod4$, $b,c,d\equiv0\pmod4$ or $a,b,c\equiv2\pmod4$, $d\equiv0\pmod4$ or $a,b,d\equiv1\pmod4$, $c\equiv3\pmod4$ or $a,c,d\equiv3\pmod4$, $b\equiv1\pmod4$, in which in each case the given pattern in Figure $3$(\ref{2-1-1}) is a VC$2$-EW for $G$. For the latter case, $a+d\equiv2\pmod4$ give a part with even number of vertices and if both parts have odd number of vertices and $a+d\equiv0\pmod4$ by the symmetry of $b$, $c$ and $a$, $d$ with the same discussion we get the desired result.

Now let $b=1, c>1$. If $e=f=0$, then $G\cong\theta(a,1,c,d)$. Thus we assume that $\{e,f\}\neq\{0\}$. In this case if $e\neq1$ or $f\neq1$, then there exists a vertex $x$ with $\deg(x)\geq 3$ such that $\deg(x)>\deg(y)$ for all $y\in N(x)$ and $G-x$ is connected. Thus by Theorem \ref{d(v)>d(u)}, $\mu(G)\leq 2$. Therefore, we assume $e=f=1$. Since $G$ is bipartite $a,c$ and $d$ are odd. For the cases $a,c,d\equiv 3\pmod4$ and $a\equiv 3\pmod4$, $c,d\equiv 1\pmod4$   and $a,c\equiv 1\pmod4$, $d\equiv 3\pmod4$ in Figure $3$(\ref{3-1}) are given a VC$2$-EW. Otherwise, $G$ has a part of even order and by Theorem \ref{delta=1},
$\mu(G)\leq 2$.

Now if $b=c=1$, then since $G$ is bipartite, $e+f$ is even and $a$, $d$ are odd.\\[.5cm]
\textbf{Case 1. $a,d\equiv 1\pmod 4$.} If $e+f\equiv 2\pmod 4$, then $G$ has a part of even number of vertices and by Theorem \ref{delta=1},
$\mu(G)\leq 2$. Thus we assume $e+f\equiv 0\pmod4$.
For the cases $e,f\equiv 0\pmod 4$ and $e,f\equiv 2\pmod 4$ we give a VC$2$-EW for $G$ in Figures $3$(\ref{1-1-1}) and $3$(\ref{1-1-2}), respectively.

For the case $e\equiv 1, f\equiv 3\pmod 4$ first we show that $\mu(G)\geq 3$. Suppose to the contrary that $G$ admits a VC$2$-EW. Since $a\equiv 1\pmod4$, by the Remark, in $p_1$ the incidence edges on $u$ and $v$ have a same weight. Similarly incidence edges on $u'$ and $v'$ in $p_4$. Thus, the incidence edges on $u$, $v$ through $p_5$ and $p_6$ must have different weights. On the other hand, since $e\equiv 1\pmod4$, two end-edges on $p_5$ get the same weight but two edges on $p_6$ get different weights, because $f\equiv 3\pmod4$. Therefore $c(u')=c(v')$. This contradiction implies $\mu(G)\geq 3$. On the other hand, $G$ is bipartite, thus is $3$-colorable and $\mu(G)\leq 3$. Therefore, $\mu(G)=3$.\\[.5cm]
\textbf{Case 2. $a,d\equiv 3\pmod 4$.} If $e+f\equiv 2\pmod 4$, then by Theorem \ref{delta=1},
$\mu(G)\leq 2$. Thus we assume $e+f\equiv 0\pmod 4$. Now one of the three possibilities $e,f\equiv 0\pmod 4$ or $e,f\equiv 2\pmod 4$ or $e\equiv 1, f\equiv 3\pmod 4$ occurs, in which in each case the given pattern in Figure $3$(\ref{1-1-1}) is a VC$2$-EW for $G$.\\[.5cm]
\textbf{Case 3. $a\equiv 1, d\equiv 3\pmod 4$.} If $e+f\equiv 0\pmod 4$, then by Theorem \ref{delta=1},
$\mu(G)\leq 2$. Thus we assume $e+f\equiv 2\pmod 4$ and for the cases $e,f\equiv1\pmod4$ or $e\equiv0\pmod4$, $f\equiv2\pmod4$ and  the case $e,f\equiv3\pmod4$ we give a VC$2$-EW for $G$ in Figures $3$(\ref{1-1-2}) and $3$(\ref{1-1-1}), respectively. Notice that in the case $e\equiv0\pmod4$, $f\equiv2\pmod4$ if $e=0$, then by Theorem \ref{d(v)>d(u)}, $\mu(G)\leq2$. Also in the case $e,f\equiv1\pmod4$ if $e=1$, we can replace weight of the edge $uv$ with $1$ to get a VC$2$-EW.

\setcounter{figure}{0}
\renewcommand{\figurename}{{}}
\begin{tabular}{|c|c|c|}
\hline
\begin{minipage}{5.09cm}
\hspace*{-.18cm}\includegraphics[scale=.26]{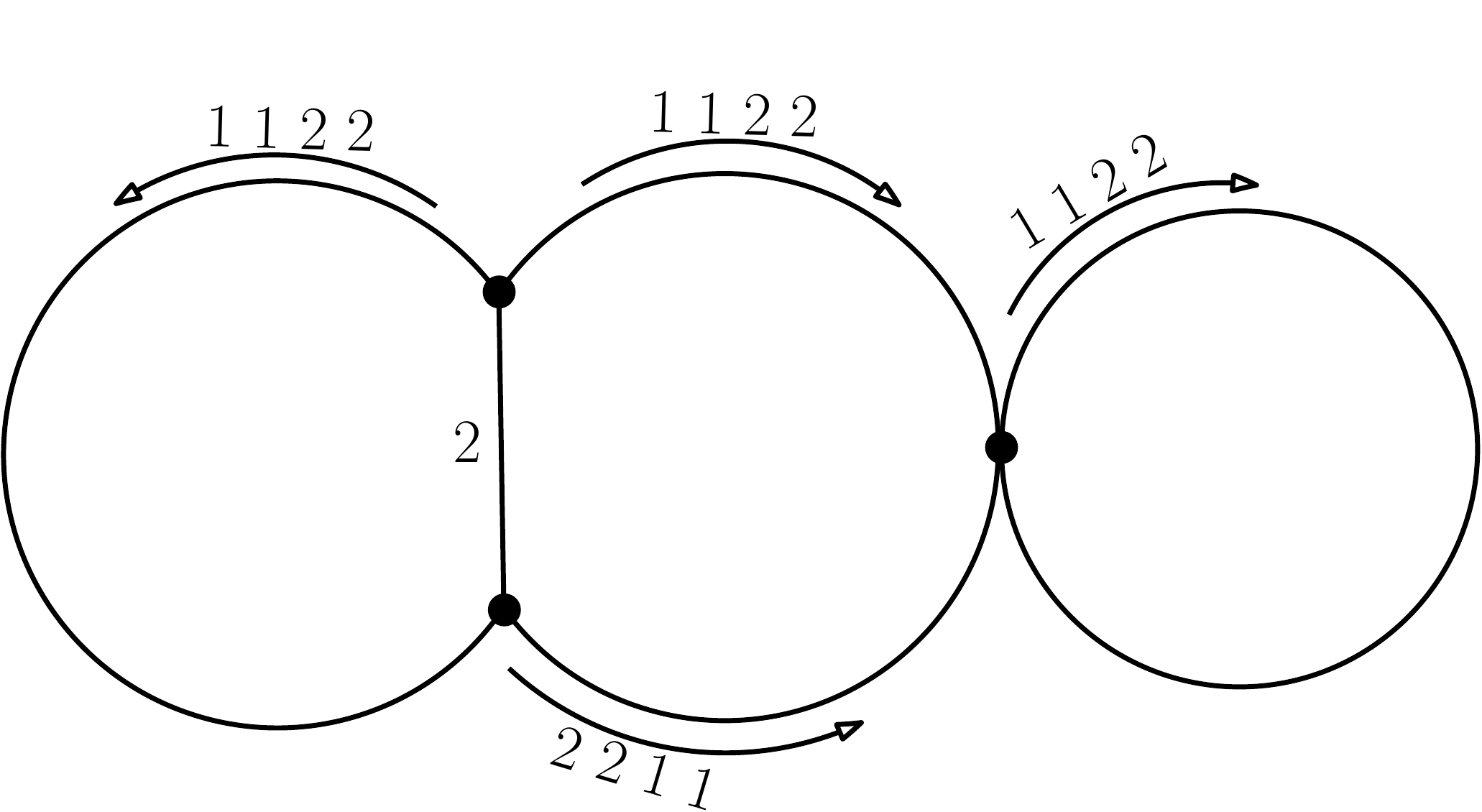}
 \captionof{figure}{\label{last}}
\end{minipage}
&
\begin{minipage}{4.32cm}
\hspace*{-.14cm}\includegraphics[scale=.26]{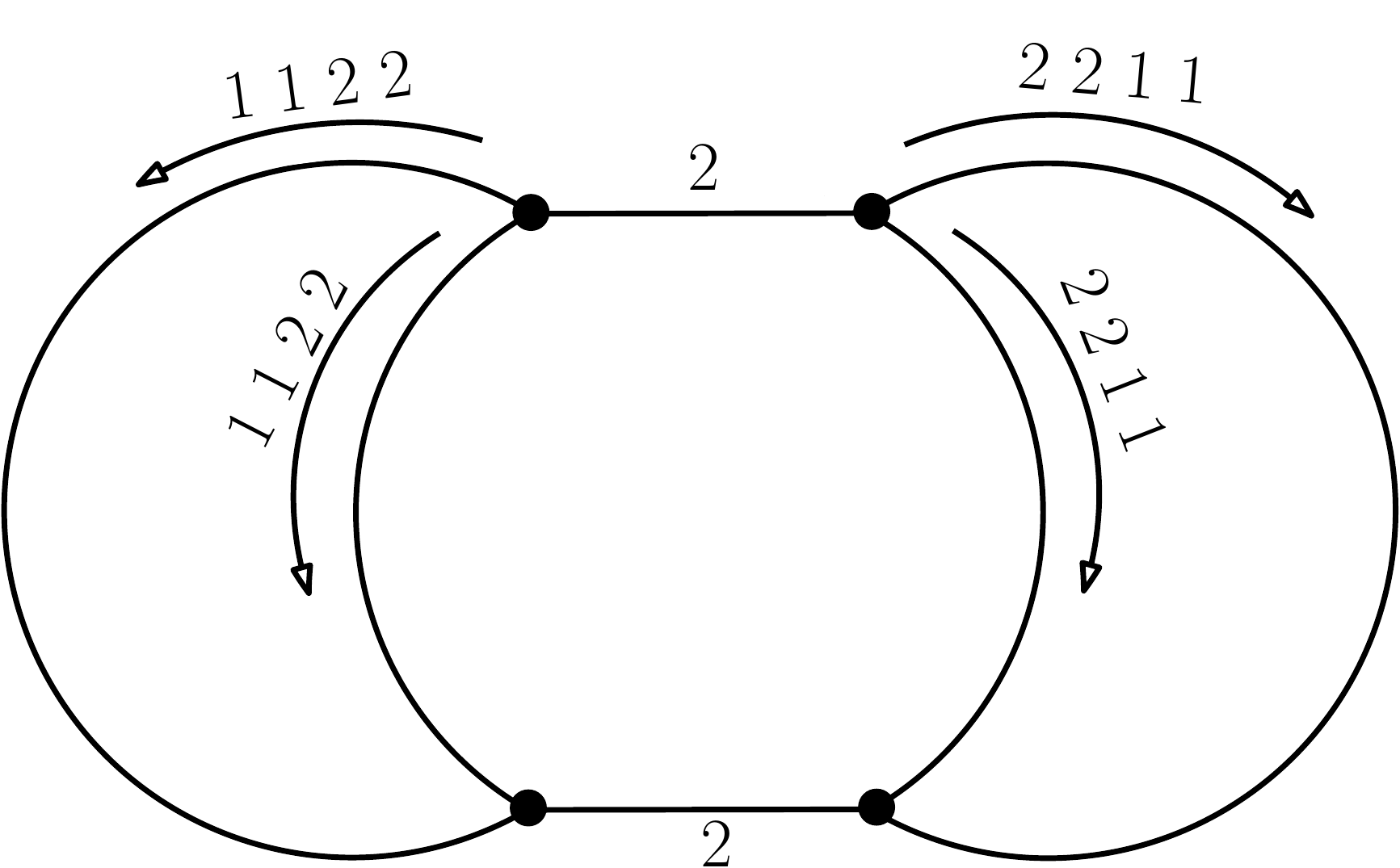}
 \captionof{figure}{\label{2-1-1}}
\end{minipage}
&
\begin{minipage}{4.3cm}
\hspace*{-.14cm}\includegraphics[scale=.26]{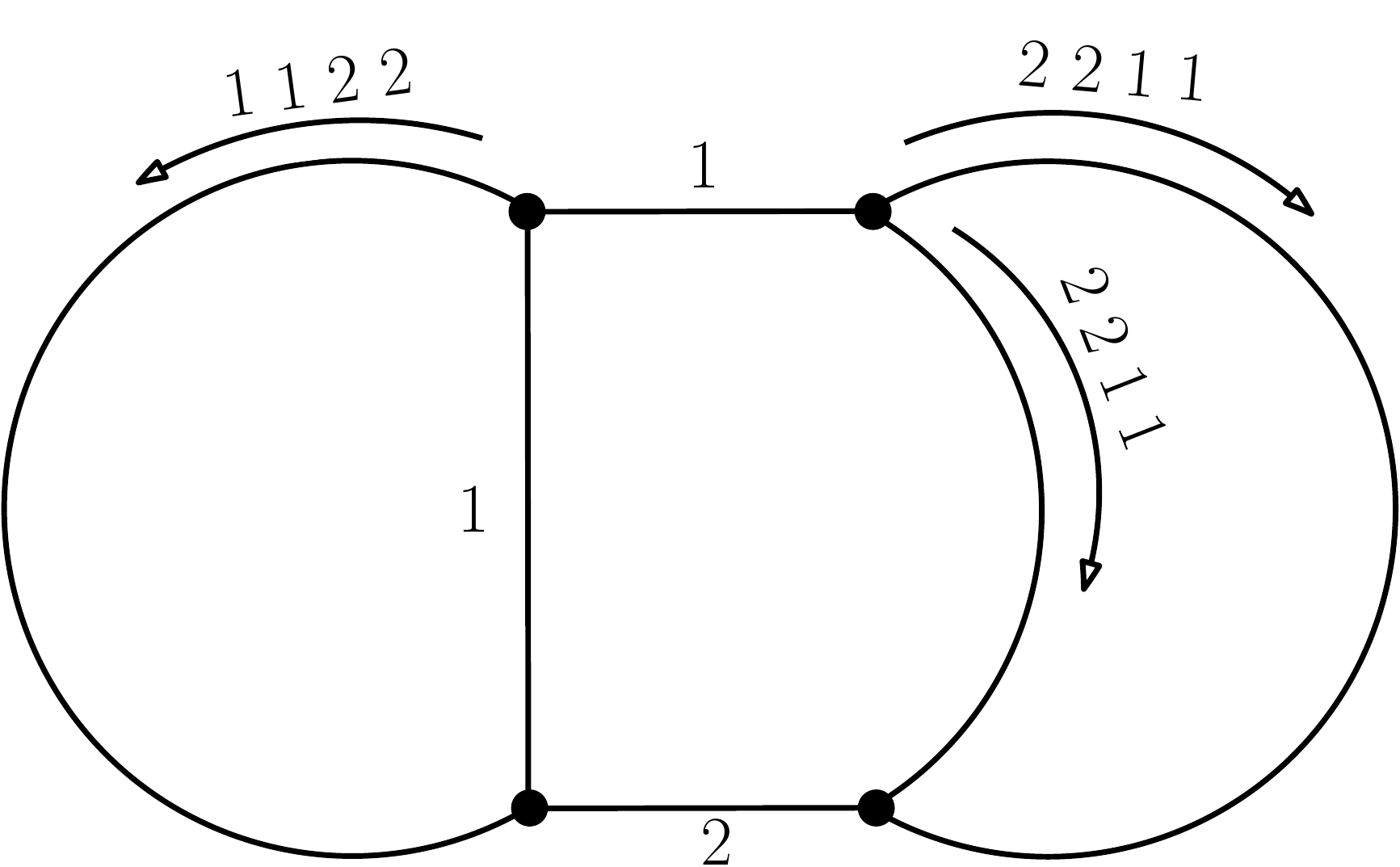}
 \captionof{figure}{\label{3-1}}
\end{minipage}
\\
\hline
\end{tabular}

\vspace*{-.3cm}
\hspace*{2.1cm}\begin{tabular}{|c|c|}
\hline
\begin{minipage}{5.09cm}
\hspace*{-.cm}\includegraphics[scale=.26]{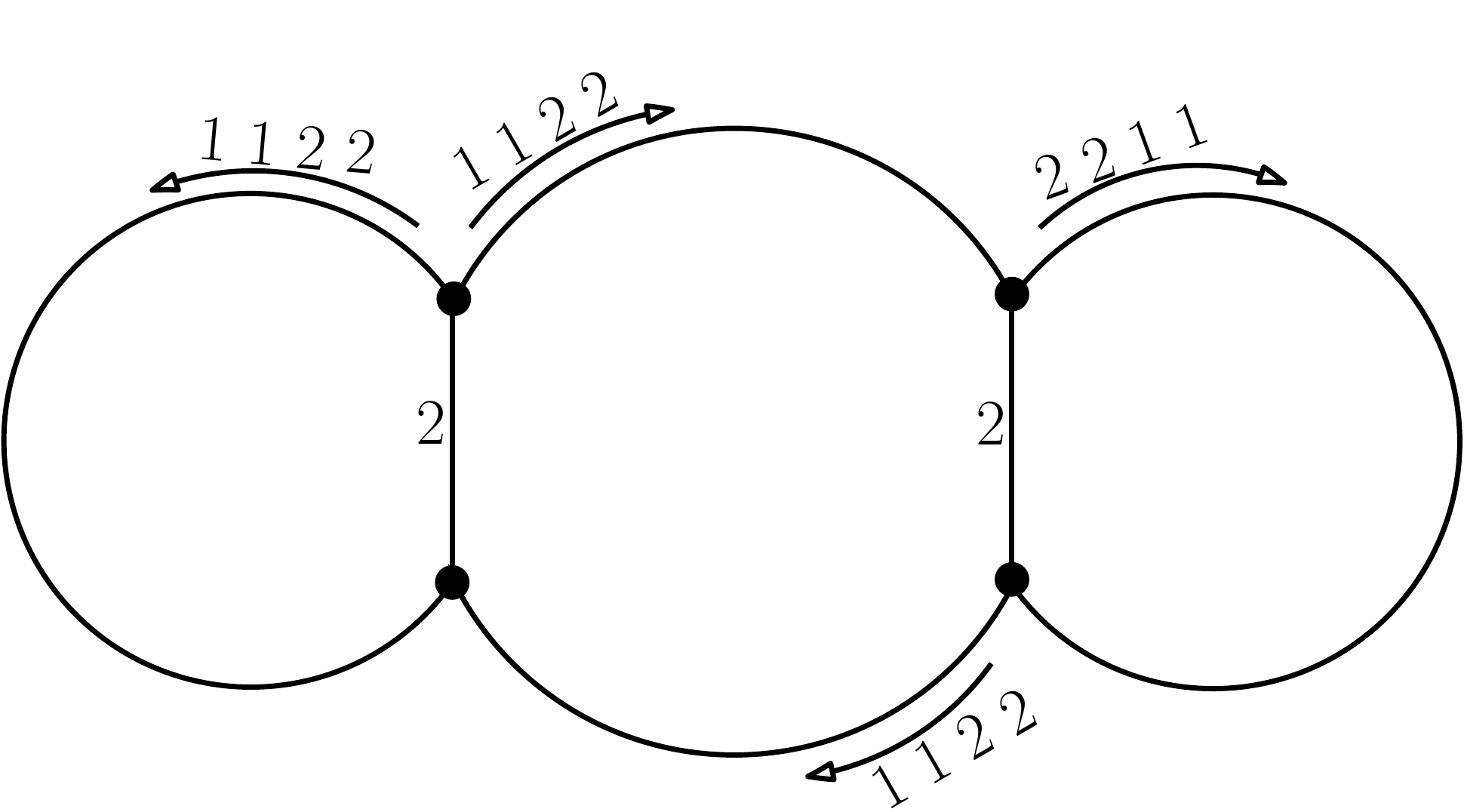}
 \captionof{figure}{\label{1-1-1}}
\end{minipage}
&
\begin{minipage}{5.1cm}
\includegraphics[scale=.26]{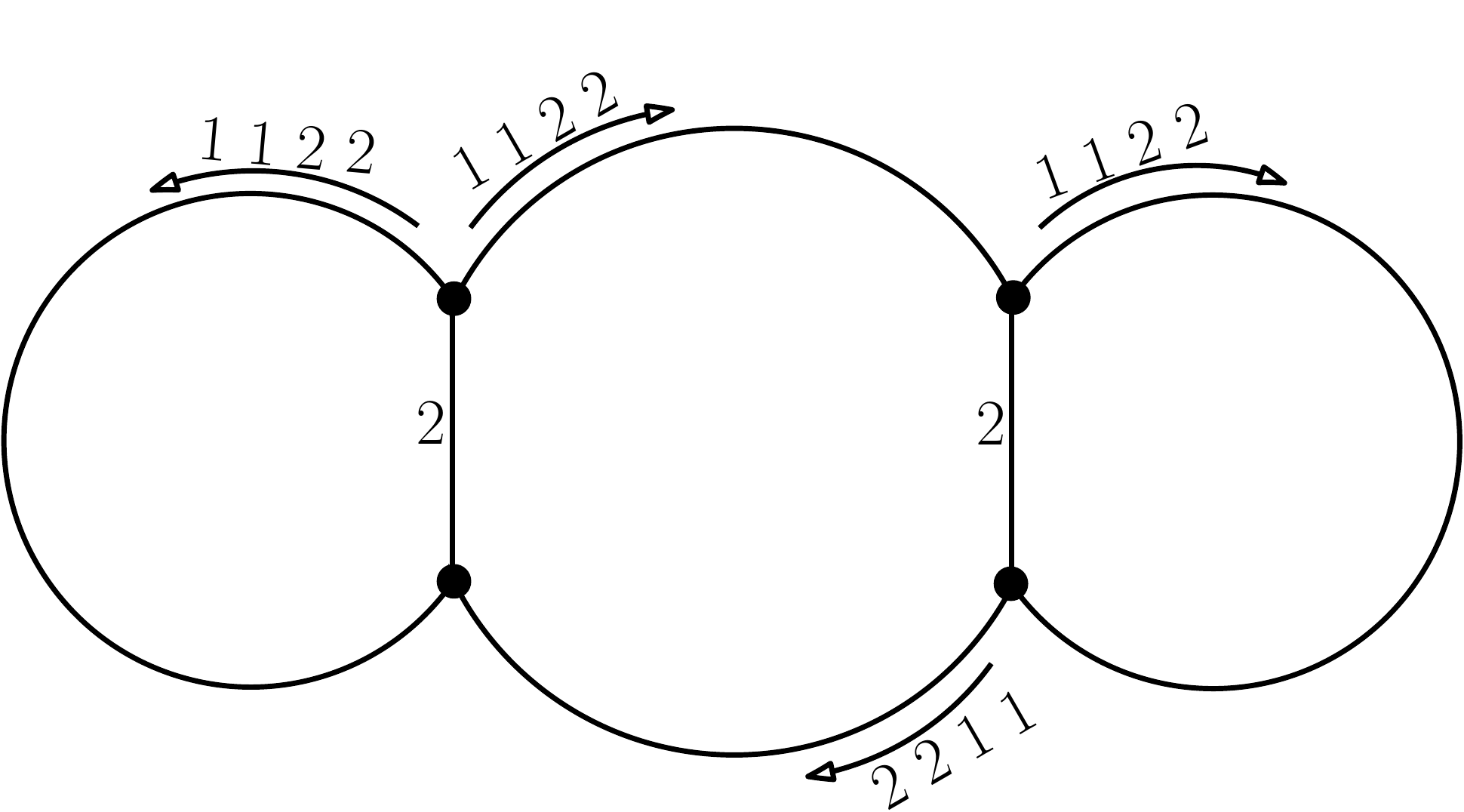}
 \captionof{figure}{\label{1-1-2}}
\end{minipage}
\\
\hline
\end{tabular}
\renewcommand{\tablename}{{Figure}}
\addtocounter{table}{2}
\captionof{table}{A VC$2$-EW for $G$}
}

\end{proof}
\section{\hspace*{-.6cm}. Conclusion}
Let $G$ be a connected graph which $G\not\cong C_{4k+r}$ for $r=1,2,3$, if $\mu(G)>2$, then by Corollary \ref{cor: MSP}, $G$ contains a maximal simple path of length $1$. In this case there are infinite families of graphs either with $\mu(G)\leq2$ (see Proposition \ref{equitable}) and $\mu(G)>2$ (see Theorems \ref{thm:GPTs} and \ref{theta}). 
If we consider the connected graphs contains no edge as a maximal simple path which don't admit a VC$2$-EW, then Corollary \ref{cor: p and p'} guarantee that there exist maximal simple paths $P$ and $P'$ where length of $P$ and length of $P'$ is $1\pmod4$ and $3\pmod4$, respectively. In Theorem \ref{thm: our conj}, we proved that every connected graph $G$ which $G\not\cong C_{4k+r}$ for $r=1,2,3$ and $\mu(G)>2$, contains an edge as a maximal simple path. In the other word, if $G\not\cong C_{4k+r}$ for $r=1,2,3$ is a connected graph contains no edge as a maximal simple path, then $\mu(G)\leq2$.

It was proved that every $3$-colorable graph admits a VC$3$-EW. Thus, for bipartite graphs $\mu(G)\leq3$, and bipartite graphs are divided into two classes, one is bipartite graphs with $\mu(G)\leq2$, say Class $1$, and the second is bipartite graphs with $\mu(G)=3$, say Class $2$. The classification of bipartite graphs in these two classes is not a trivial problem.

It is proved that every $3$-connected bipartite graph is in Class $1$. For a graph contains cut vertices, some properties of it's blocks which guarantee to $G$ be in Class $1$ are given (Theorem \ref{2 Blocks} and Corollary \ref{k Blocks}). Corollary \ref{cor: MSP} concludes that every $2$-connected bipartite graph which has no ear of length $1$ in it's ear decomposition belongs to Class $1$.  In Theorem \ref{thm:GPTs} and Example \ref{exm:mu=3} infinite families of $2$-connected bipartite graphs in Class $2$ are provided. 

\end{document}